\documentclass[11pt]{amsart}
\usepackage{bm}
\usepackage{fullpage}
\usepackage{amssymb}
\usepackage{amsmath,amsfonts,amsthm}
\usepackage{hyperref}
\usepackage{color}
\usepackage{cite}
\usepackage{tikz}
\definecolor{uuuuuu}{rgb}{0.26666666666666666,0.26666666666666666,0.26666666666666666}

\newtheorem{theorem}{Theorem}[section]

\newtheorem{lemma}[theorem]{Lemma}

\newtheorem{claim}[theorem]{Claim}

\title{The $m$-bipartite Ramsey number of the $K_{2,2}$ versus $K_{6,6}$}
\author[1]{Yaser Rowshan   }


\email{y.rowshan@iasbs.ac.ir,~~~y.rowshan.math@gmail.com}

\begin{document}
	\maketitle
	
 	\begin{abstract} Given bipartite graphs $G_1, \ldots, G_n$, the bipartite Ramsey number $BR(G_1,\ldots, G_n)$ is the last integer $b$ such
 		that any complete bipartite graph $K_{b,b}$ with edges coloured with colours $1,2,\ldots,n$ contains a copy of some $G_i$ ($1\leq i\leq n$) where all edges of $G_i$ have colour $i$.   As another view of bipartite Ramsey numbers, for given bipartite graphs $G_1, \ldots, G_n$ and a positive integer $m$, the $m$-bipartite Ramsey number $BR_m(G_1, \ldots, G_n)$, is defined as the least integer $b$, such that any complete bipartite graph $K_{m,b}$ with edges coloured with colours $1,2,\ldots,n$ contains a copy of some $G_i$ ($1\leq i\leq n$) where all edges of $G_i$ have colour $i$. The size of $BR_m(G_1, G_2)$, where $G_1=K_{2,2}$ and  $G_2\in \{K_{3,3}, K_{4,4}\}$ for each $m$ and  the size of   $BR_m(K_{3,3}, K_{3,3})$ and $BR_m(K_{2,2}, K_{5,5})$ for special values of $m$, have been determined in several article up to now. In this article,  we compute the size  of  $BR_m(K_{2,2}, K_{6,6})$ for some $m\geq 2$.
	\end{abstract}
	
	\section{Introduction}
Extremal graph theory problems generally  ask for the max/ min order or size of a graph having certain characteristics. Such problems are often quite natural in the construction of networks or circuits. Ramsey theory explores the question of how big a structure must be to contain a certain substructure or substructures. The  Ramsey number $R(G, H)$ is the smallest order of a complete graph, so that any $2$-coloring of the edges must result in either a copy of graph $G$ in the first color or a copy of graph $H$ in the second color.  It is shown that  $R(G, H)\leq R(K_m, K_n)$, where $G$ and $H$ be two arbitrary graph of size $m$ and $n$, respectively.

 Bipartite Ramsey problems deal with the same questions but the graph explored is the complete bipartite graph instead of the complete graph. Given bipartite graphs $G_1, \ldots, G_n$, the bipartite Ramsey number $BR(G_1,\ldots, G_n)$ is the last integer $b$ such that any complete bipartite graph $K_{b,b}$ with edges coloured with colours $1,2,\ldots,n$ contains a copy of some $G_i$ ($1\leq i\leq n$) where all edges of $G_i$ have colour $i$. One can refer to \cite{goedgebeur2022new, hattingh1998star, bucic2019multicolour, bucic20193, lakshmi2020three}, \cite{wang2021bipartite, hatala2021new, math10050701,  goddard2000bipartite} and their references for further studies.
 
 As new view of bipartite Ramsey numbers, for given bipartite graphs $G_1, \ldots, G_n$ and a positive integer $m$, the $m$-bipartite Ramsey number $BR_m(G_1, \ldots, G_n)$, is defined as the least integer $b$, such that any complete bipartite graph $K_{m,b}$ with edges coloured with colours $1,2,\ldots,n$ contains a copy of some $G_i$ ($1\leq i\leq n$) where all edges of $G_i$ have colour $i$. The size of $BR_m(G_1, G_2)$ where $G_1\in \{K_{2,2}, K_{3,3}\}$ and  $G_2\in \{K_{3,3}, K_{4,4}, K_{5,5}\}$, have been determined in several papers up to now. One can refer to \cite{bi2018another, 2022arXiv220112844R, chartrand2021new, bi2019new, rowshan2022m1, Rowshan2022TheR} and their references for further studies. In particular, The following results have been obtained on $m$-bipartite Ramsey numbers.
	\begin{theorem}\cite{bi2018another, 2022arXiv220112844R}
Let $m\geq 2 $, then:
	\[
	BR_m(K_{2,2},K_{3,3})= \left\lbrace
	\begin{array}{ll}	
		\text{does not exist}, & ~~~where~~~m=2,3,\vspace{.2 cm}\\
		15& ~~~where~~~m=4,\vspace{.2 cm}\\
		12& ~~~where~~~m=5,6,\vspace{.2 cm}\\
		9 & ~~~where~~~ m=7,8.\vspace{.2 cm}\\
		
	\end{array}
	\right.
	\]	
\end{theorem} 	

	\begin{theorem}\cite{ bi2019new}
	For each $m\in \{1,2,\ldots,9\}$, we have:
	\[
	BR_m(K_{3,3},K_{3,3})= \left\lbrace
	\begin{array}{ll}	
		\text{does not exist}, & ~~~where~~m=2,3,4,\vspace{.2 cm}\\
		41 & ~~~where~~~m=5,6,\vspace{.2 cm}\\
		29 & ~~~where~~~ m=7,8.\vspace{.2 cm}\\
		
	\end{array}
	\right.
	\]	
\end{theorem} 
	\begin{theorem}\cite{Rowshan2022TheR}
	For each $m\in \{1,2,\ldots,8\}$, we have:
	\[
	BR_m(K_{2,2},K_{5,5})= \left\lbrace
	\begin{array}{ll}	
		\text{does not exist}, & ~~~where~~m=2,3,4,5,\vspace{.2 cm}\\
		40 & ~~~where~~~m=6,\vspace{.2 cm}\\
	30 & ~~~where~~~ m=7,8.\vspace{.2 cm}\\
		
	\end{array}
	\right.
	\]	
\end{theorem} 
	In this paper,  we compute the exact value of $BR_m(K_{2,2}, K_{6,6})$ for some $m\geq 2$ as follow.
	\begin{theorem}\label{M.th}[Main theorem]
		For each $m\in \{1,2,\ldots,8\}$, we have:
		\[
		BR_m(K_{2,2},K_{6,6})= \left\lbrace
		\begin{array}{ll}	
			\text{does not exist}, & ~~~where~~m=2,3,4,5,6,\vspace{.2 cm}\\
		57 & ~~~where~~~m=7,\vspace{.2 cm}\\
		45 & ~~~where~~~ m=8.\vspace{.2 cm}\\
		 
		\end{array}
		\right.
		\]	
	\end{theorem} 

 Suppose  that $G[X, Y]$, is a bipartite graph with bipartition sets $X$ and $Y$. Let $E(G[ X',Y'])$, denotes the edge set of $G[ X', Y']$, where $X'\subseteq X$ and $Y'\subseteq Y$. We use $\Delta (G_X)$ and $\Delta (G_{Y})$ to denote the maximum degree of  vertices in  part $X$ and $Y$ of $G$, respectively. The degree of a vertex $v\in V(G)$, is denoted by $\deg_G(v)$. For each $v\in X(Y)$, $N_G(v) =\{ u\in Y(X), ~~vu\in E(G)\}$. For given  graphs $G$, $H$, and $F$, we say $G$ is $2$-colorable to $(H, F)$, if there is a  subgraph  of $G$ say $G'$, where $H\nsubseteq G'$ and $F\nsubseteq \overline{G'}$. We use $G\rightarrow (H, F)$, to show that  $G$  is $2$-colorable to $(H, F)$.  To simplify we use $[n]=\{1,2,\ldots, n\}$.

	\section{\bf Proof of the main results}
In this section, we establish the main results of the paper. Before that, we give a lemma as follow, which help to prove Theorem \ref{M.th}.  
\begin{lemma}\label{l1} Suppose  that $(X=\{x_1,\ldots,x_m\},Y=\{y_1,\ldots ,y_{n}\})$, where $m\geq 7$ and $n\geq 12$ are the partition  sets of  $K=K_{m,n}$.	Let $G$ is a subgraph of $K_{m,n}$. If $\Delta (G_X)\geq 12$, then either $K_{2,2} \subseteq G$ or $K_{6,6} \subseteq \overline{G}$.
\end{lemma}
\begin{proof}
	Without loss of generality	(W.l.g), let $\Delta (G_X)=12$ and  $N_G(x)=Y'$, where  $|Y'|=12$ and $K_{2,2} \nsubseteq G$.  Therefore, 	 $|N_G(x')\cap Y'|\leq 1$ for each $x'\in X\setminus\{x\}$. Since $|X|\geq 7$ and $|Y'|= 12$, so one can check that  $K_{6,6} \subseteq \overline{G}[X\setminus\{x\}, Y']$.
\end{proof}
To prove   Theorem \ref{M.th}, we start with  the following theorem. 
\begin{theorem}\label{th1} For each $m\in \{2,3,4,5,6\}$, the number $BR_m(K_{2,2}, K_{6,6})$ does not exist.
 
\end{theorem}
\begin{proof}
Suppose that $m\in \{2,3,4,5,6\}$. For an arbitrary integer $t \geq 6$, set $K=K_{m,t}$ and let $G$ be a subgraph of $K$, such that $G= K_{1,t}$. Therefore, we have $\overline{G}= K_{m-1,t}$. Hence, one can checked that, neither a $K_{2,2}\subseteq G$ nor a  $K_{6,6}\subseteq \overline{G}$. Therefore for each $m\in \{2,3,4,5,6\}$, the number $BR_m(K_{2,2}, K_{6,6})$ does not exist.
\end{proof}
In the next results, we determined the size of $BR_m(K_{2,2}, K_{6,6})$ for $m=7$.
\begin{theorem}\label{th2}
		$BR_7(K_{2,2}, K_{6,6})=57$.
\end{theorem}
\begin{proof}
Suppose  that $(X=\{x_1,\ldots,x_7\},Y=\{y_1,y_2,\ldots ,y_{56}\})$ are the partition  sets of  $K=K_{7,56}$. Let $G\subseteq K$, such that   $N_G(x_i)=Y_i$, with the following properties:
\begin{itemize}
		 	
	   \item{\bf (A1):} $Y_1=\{y_1,y_2,\ldots, y_{11}\}$.	
			\item{\bf (A2):} $Y_2=\{y_1,y_{12}, y_{11}\ldots, y_{21}\}$.
			\item{\bf (A3):} $Y_3=\{y_2,y_{12}, y_{22},y_{23}\ldots, y_{30}\}$.
			\item{\bf (A4):} $Y_4=\{y_3,y_{13}, y_{22}, y_{31},\ldots, y_{38}\}$.
			\item{\bf (A5):} $Y_5=\{y_4,y_{14}, y_{23}, y_{31}, y_{39}, y_{40},\ldots, y_{45}\}$.
			\item{\bf (A6):} $Y_6=\{y_5,y_{15}, y_{24}, y_{32}, y_{39}, y_{46},y_{47},\ldots, y_{51}\}$.
			\item{\bf (A7):} $Y_7=\{y_6,y_{16}, y_{25}, y_{33}, y_{40}, y_{46},y_{52},\ldots, y_{56}\}$.
		\end{itemize}
		So, for each $i,j\in[6]$, by $(Ai)$ and $(Aj)$, one can  checked that  $|N_G(x_i)\cap N_G(x_j)|=1$,  and $|\cup_{j=1, j\neq i}^{j=7} N_G(x_j)|=51$. Therefore, $K_{2,2}\nsubseteq G$ and  $K_{6,6}\nsubseteq \overline{G}[X\setminus\{x_i\}, Y]$ for each $i\in[7]$. Hence,  $BR_7(K_{2,2}, K_{6,6})\geq 57$. 
		
		Now, suppose that $(X=\{x_1,\ldots, x_7\},Y=\{y_1,\ldots ,y_{57}\})$ are the partition  sets of  $K=K_{7,57}$. Suppose that $G$ be a subgraph of $K$, so that $K_{2,2} \nsubseteq G$. Consider $\Delta=\Delta (G_X)$. One can suppose that $\Delta\in \{9,10,11\}$. Otherwise,  if $\Delta\geq 12$, then theorem holds by Lemma \ref{l1}. Also  for the case that $\Delta\leq 8$, it is clear that $K_{6,6}\subseteq \overline{G}$. Now, we the following claims.
		\begin{claim}\label{c1}
			If	$ \Delta=9$, then  $K_{6,6}\subseteq \overline{G}$.
		\end{claim}	
		\begin{proof}[Proof of Claim \ref{c1}]
			W.l.g let $N_G(x_1)=Y_1=\{y_1,\ldots,y_{9}\}$. As $K_{2,2} \nsubseteq G$,  we have $|N_G(x_i)\cap N_G(x_j)|\leq 1$ for each $i,j\in [7]$. Also, it can be checked that $|N_G(x)\cap Y_1|= 1$ for at least four members of $X\setminus \{x_1\}$, otherwise $K_{6,6}\subseteq \overline{G}[X,Y_1]$. W.l.g  let $|N_G(x)\cap Y_1|= 1$ for each $x_i\in  \{x_2,x_3,x_4,x_5\}$. Therefore, as $ \Delta=9$ and $|N_G(x)\cap Y_1|= 1$ for each $x_i\in  \{x_2,x_3,x_4,x_5\}$, then one can check that  $|\cup_{j=1}^{j=5} N_G(x_j)|\leq 41$. Hence, as $ N_G(x_6)\leq \Delta=9$, then we have $|\cup_{j=1}^{j=6} N_G(x_j)|\leq 50$. Therefore as $|Y|=57$,  we have $K_{6,6}\subseteq \overline{G}[X\setminus\{x_7\}, Y]$,  hence the claim holds.    
		\end{proof}
	\begin{claim}\label{c2}
	If	$ \Delta=10$, then  $K_{6,6}\subseteq \overline{G}$.
\end{claim}	
\begin{proof}[Proof of Claim \ref{c2}]
	W.l.g let $N_G(x_1)=Y_1=\{y_1,\ldots,y_{10}\}$. As $K_{2,2} \nsubseteq G$,  we have $|N_G(x_i)\cap N_G(x_j)|\leq 1$ for each $i,j\in [7]$. Also, it can be checked that $|N_G(x)\cap Y_1|= 1$ for at least five members of $X\setminus \{x_1\}$, otherwise $K_{6,6}\subseteq \overline{G}[X,Y_1]$. W.l.g  let $|N_G(x)\cap Y_1|= 1$ for each $x\in  X'=\{x_2,x_3,x_4,x_5, x_6\}$. Therefore, as $ \Delta=10$ and $|N_G(x)\cap Y_1|= 1$ for each $x\in X'$, then one can check that  $|\cup_{j=1}^{j=5} N_G(x_j)|\leq 55$. If  $|N_G(x)\cap (Y\setminus Y_1)|\leq 8$ for at least four members of  $X'$, then one can check that  $|\cup_{j=1}^{j=6} N_G(x_j)|\leq 51$. Therefore,  as $|Y|=57$,  we have $K_{6,6}\subseteq \overline{G}[X\setminus\{x_7\}, Y]$, that is the claim is true. Now, suppose that $|N_G(x)\cap (Y\setminus Y_1)|= 9$ for at least two members of  $X'$. W.l.g let $|N_G(x)\cap (Y\setminus Y_1)|= 9$ for each $x\in \{x_2,x_3\}$. For $i=2,3$, we may suppose that  $N_G(x_i)\cap (Y\setminus Y_1)=Y_i$. As $|Y_i|=9$, if $|Y_2\cap Y_3|=0$, then it is easy to check that $|N_G(x)\cap Y_i|= 1$ for each $x\in \{x_4,x_5,x_6\}$. Otherwise, one can say that $K_{6,6}\subseteq \overline{G}[X\setminus \{x_i\},Y_i]$ for some $i\in \{2,3\}$. So, we have $|\cup_{j=1}^{j=6} N_G(x_j)|\leq 10+9+9+7+7+7=49$, hence the proof is the same.  So, let   $|Y_2\cap Y_3|=1$. Therefore, we have $|\cup_{j=1}^{j=3} N_G(x_j)|=27$. If $|N_G(x)\cap Y_i|= 1$ for each $x\in \{x_4,x_5,x_6\}$, then the proof is the same.  Now, for each $j=4,5,6$, let   $|N_G(x_j)\cap Y_i|=1$ for at least one $i\in \{2,3\}$.   Therefore, we have $|N_G(x)\cap (Y_1\cup Y_2\cup Y_3)|\geq 2$, that is  $|N_G(x)\cap (Y\setminus Y_1\cup Y_2\cup Y_3)|\leq 8$ for each $x\in \{x_4, x_5,x_6\}$. So, we have $|\cup_{j=1}^{j=6} N_G(x_j)|\leq 10+9+8+8+8+8=51$. Therefore, we have	$K_{6,6}\subseteq \overline{G}[X\setminus\{x_7\}, Y]$. Now,  let   there is a member of  $\{x_4,x_5,x_6\}$ say $x$, such that $|N_G(x)\cap Y_i|= 0$ for each $i=2,3$,  w.l.g let  $x=x_4$. As $|Y_3|=9$, $|Y_2\cap Y_3|=1$, and $|N_G(x_4)\cap Y_i|= 0$, it is easy to check that $|N_G(x)\cap Y_i|= 1$ for each $x\in \{x_5,x_6\}$ and each $i\in \{2,3\}$. Otherwise, one can say that $K_{6,6}\subseteq \overline{G}[X\setminus \{x_i\},Y_i]$ for some $i=2,3$. So, we have $|\cup_{j=1, j\neq 4}^{j=6} N_G(x_j)|\leq 10+9+8+7+7=41$. Therefore, as $\Delta =10$, $|N_G(x_4)\cap Y_1|= 1$,   we have $|\cup_{j=1}^{j=6} N_G(x_j)|\leq 50$, that is	$K_{6,6}\subseteq \overline{G}[X\setminus\{x_7\}, Y]$. Hence claim holds.	
\end{proof}
Therefore, by Claims \ref{c1}, \ref{c2}, let $ \Delta=11$.   W.l.g let  $|N_G(x_1)|=11$ and $Y_1=N_G(x_1)=\{y_1,\ldots,y_{11}\}$. Now, we have the following claim.

	\begin{claim}\label{c3}
If either $|N_G(x_i)\cap Y_1|= 0$ or $N_G(x_i)\cap Y_1 = N_G(x_j)\cap Y_1$ for some $i,j\in\{2,\ldots,7\}$, then $K_{6,6}\subseteq \overline{G}$.
	\end{claim}	
	\begin{proof}[Proof of Claim \ref{c3}]
	 As $K_{2,2} \nsubseteq G$, so $|N_G(x_i)\cap Y_1|\leq 1$ for each $i$.  Now, let  $|N_G(x_2)\cap Y_1|= 0$.  Therefore, it is clear that $K_{6,6}\subseteq \overline{G}[X\setminus\{x_1\},Y_1]$. Also, w.l.g let  $N_G(x_2)\cap N_G(x_3)\cap Y_1=\{y\}$, then as $|X|=7$ and  $|Y_1|=11$, we have $K_{6,6}\subseteq \overline{G}[X\setminus \{x_1\},Y_1\setminus \{y\}]$.	
	\end{proof}	
Therefore, by  Claim \ref{c3}  it is clear  that  $K_{6,5}\subseteq \overline{G}[X\setminus \{x_1\},Y_1]$.  If  $|N_G(x)|=11$ for each $x\in X$, then  by  Claim \ref{c3}  we have  $|\cup_{j=1}^{j=6} N_G(x_j)|= 56$, that is there exists a member of $Y$ say $y_{57}$, such that $|N_G(y_{57})|=0$.   Therefore, we have $K_{6,6}\subseteq \overline{G}[X\setminus \{x_1\},Y_1\cup \{y_{57}\}]$ and the proof is complete.

 So, assume that  $|N_G(x)|\leq 10$ for at least one member of  $X\setminus \{x_1\}$. 	W.l.g let $\deg_G(x_2)=10$ and $ N_G(x_2)=Y_2=\{y_1, y_{12}, y_{13}\ldots,y_{20}\}$. One can suppose that there is at least four members of $X\setminus \{x_1,x_2\}$, say $X'=\{x_3,x_4,x_5,x_6\}$, so that for each $i,j\in \{3,4,5,6\}$, $|N_G(x_i)\cap Y_2\setminus \{y_1\}|= 1$ and $N_G(x_i)\cap Y_2\setminus \{y_1\} \neq N_G(x_j)\cap Y_2\setminus \{y_1\}$. Otherwise, by an argument similar to the proof of  Claim \ref{c3}, we have $K_{6,6}\subseteq \overline{G}[X\setminus \{x_2\},Y_2]$. Therefore, as $ \Delta=11$, and $|N_G(x_i)\cap Y_i\setminus \{y_1\}|= 1$ for each $x\in X'$ and each $i=1,2$, one can check that  $|\cup_{j=1}^{j=6} N_G(x_j)|\leq 11+9+9+9+9+9=56$. Hence, the proof is the same as the case that $|N_G(x)|=11$ for each $x\in X$. So, the theorem holds.
\end{proof}
	In the next results, we determined the size of $BR_8(K_{2,2}, K_{6,6})$.
	\begin{theorem}\label{th3}
		$BR_8(K_{2,2}, K_{6,6})=45$.
	\end{theorem}
	\begin{proof}
	 Let $(X=\{x_1,\ldots, x_8\},Y=\{y_1,\ldots ,y_{44}\})$ are the partition  sets of  $K=K_{8,44}$. Suppose that $G$ be a subgraph of $K$ such that for each $i\in[8]$,  $N_G(x_i)=Y_i$ with the following properties.
		\begin{itemize}
			
			\item{\bf (B1):} $Y_1=\{y_1,\ldots, y_9\}$.	
			\item{\bf (B2):} $Y_2=\{y_1, y_{10},y_{11},\ldots, y_{17}\}$.
			\item{\bf (B3):} $Y_3=\{y_2,y_{10}, y_{18},y_{19},\ldots, y_{24}\}$.
			\item{\bf (B4):} $Y_4=\{y_3,y_{11}, y_{18}, y_{25},y_{26}, y_{27}, y_{28},y_{29}, y_{30}\}$.
			\item{\bf (B5):} $Y_5=\{y_4,y_{12}, y_{19}, y_{25}, y_{31}, y_{32},y_{33}, y_{34}, y_{35}\}$.
			\item{\bf (B6):} $Y_6=\{y_5,y_{13}, y_{20}, y_{26}, y_{31}, y_{36},y_{37},y_{38},y_{39}\}$.
			\item{\bf (B7):} $Y_7=\{y_6,y_{14}, y_{21}, y_{27}, y_{32}, y_{36},y_{40},y_{41},y_{42}\}$.
			\item{\bf (B8):} $Y_8=\{y_7,y_{15}, y_{22}, y_{28}, y_{33}, y_{37},y_{40},y_{43},y_{44}\}$.
		\end{itemize}
	By considering $(Bi)$ and $(Bj)$, it can be said that:
		\begin{itemize}
			
			\item{\bf (C1):} $|N_G(x_i)\cap N_G(x_j)|=1$, for each $i,j\in [8]$.
			\item{\bf (C2):} $|\cup_{i=1}^{i=6} N_G(x_{j_i})|= 39$, for each $j_1,\ldots,j_6\in [8]$.
		\end{itemize}
		
	Therefore	by $(C1)$,  we have $K_{2,2}\nsubseteq G$. Also, by $(C2)$, one can check that  $K_{6,6}\nsubseteq \overline{G}$, which means that  $K_{8,44} \rightarrow (K_{2,2},K_{6,6})$. Therefore, the lower bound holds.

		Now, we prove the upper bound. We may suppose that  $(X=\{x_1,\ldots,x_8\},Y=\{y_1,\ldots ,y_{45}\})$ are the partition  sets of  $K=K_{8,45}$. Let $G$ be a subgraph of $K$, where $K_{2,2} \nsubseteq G$. We show that $K_{6,6}\subseteq \overline{G}$.  Consider $\Delta=\Delta (G_X)$. As $K_{2,2} \nsubseteq G$, then by Lemma \ref{l1}, one can assume that $\Delta\leq 11$. Hence:
	
		\begin{claim}\label{c4}
		If	$ \Delta= 11$, then $K_{6,6}\subseteq \overline{G}$.
		\end{claim}	
		\begin{proof}[Proof of Claim \ref{c4}]
		 W.l.g let $|N_G(x_1)=Y_1|=11$. Since $K_{2,2} \nsubseteq G$, $|X|=8$, and $|Y_1|=11$, then  one can suppose that $|N_G(x_i)\cap Y_1|= 1$, and  for each $i,j\in\{2,\ldots,8\}$, $x_i$ and $x_j$ have a different neighborhood in $Y_1$. Otherwise, in any case it is clear that $K_{6,6}\subseteq \overline{G}[X,Y_1]$. Therefore,   for each $x\neq x_1$, we have  $K_{6,5}\subseteq \overline{G}[X\setminus \{x_1,x\},Y_1]$. If there is a member of $Y\setminus Y_1$ say $y$, such that $|N_{\overline{G}}(y)\cap (X\setminus\{x_1\})|\geq 6$, then  $K_{6,6}\subseteq \overline{G}[X\setminus \{x_1\}, Y_1\cup \{y\}]$. Hence, let $|N_G(y)\cap (X\setminus\{x_1\}) |\geq 2$ for each $y\in Y\setminus Y_1$. Therefore,  $|E(G[X\setminus \{x_1\}, Y\setminus Y_1])|\geq 34\times 2=68$. Hence, by pigeon-hole principle, there is at least one member of $X\setminus\{x_1\}$ say $x_2$, such that $|N_G(x_2)\cap (Y\setminus Y_1)|\geq 10$. Set $N_G(x_2)\cap (Y\setminus Y_1)=Y_2$. Now, as $K_{2,2} \nsubseteq G$, then $|N_G(x_i)\cap Y_2|\leq 1$ for each  $i\in \{3,\ldots,8\}$. Therefore, since $|Y_2|\geq 10$, one can checked that $K_{6,1}\subseteq \overline{G}[X\setminus\{x_1,x_2\}, Y_2]$. Hence,  as $K_{6,5}\subseteq \overline{G}[X\setminus \{x_1,x_2\},Y_1]$, we have $K_{6,6}\subseteq \overline{G}[X\setminus\{x_1,x_2\}, Y_1\cup Y_2]$. So, the claim holds.
		\end{proof}
Therefore,	by Claim  \ref{c4}, one can assume that $\Delta\leq 10$. Now, we have the next claim.	
	\begin{claim}\label{c5}
		If there exist  $V\subseteq X$, such that $|V|=5$ and $|\cup_{x\in V}N_G(x)|\leq 35$, then we have $K_{6,6}\subseteq \overline{G}$.
	\end{claim}	
	
	\begin{proof}[Proof of Claim \ref{c5}]
		W.l.g assume that  $V=\{x_1,x_2,x_3,x_4,x_5\}$ and let $Y'=\cup_{x\in V}N_G(x)$ where $|Y'|\leq 35$. If  $|Y'|\leq 29$, then  as $\Delta\leq 10$, it is easy to said that $|Y'\cup N_G(x)|\leq 39$ for each $x\in X\setminus V$, that is $K_{6,6}\subseteq \overline{G}[V\cup \{x\}, Y]$, hence the claim holds. So, suppose that $|Y'|\in \{30,\ldots,34,35\}$. Assume that $|Y'|=35$.  Hence,  for each $x\in \{x_6,x_7,x_8\}$, one can assume that  $|N_G(x)\cap Y\setminus Y'|\geq 5$, otherwise we have  $|Y'\cup N_G(x)|\leq 39$ for some $x\in X\setminus V$, that is $K_{6,6}\subseteq \overline{G}$. Therefore, as $|Y\setminus \cup_{i=1}^{i=5} N_G(x_{i})|=10$ and $|\{x_6,x_7,x_8\}|=3$, it easy to say that  $K_{2,2}\subseteq G[\{x_6,x_7,x_8\}, Y\setminus \cup_{i=1}^{i=5} N_G(x_{i})]$, a contradiction.  Now let $|Y'|=34$.  Hence,  for each $x\in \{x_6,x_7,x_8\}$, one can assume that  $|N_G(x)\cap Y\setminus Y'|\geq 6$, otherwise  $|Y'\cup N_G(x)|\leq 39$ for some $x\in X\setminus V$, that is we have $K_{6,6}\subseteq \overline{G}$. Therefore, as $|Y\setminus \cup_{i=1}^{i=5} N_G(x_{i})|=11$ and $|\{x_6,x_7,x_8\}|=3$, it easy to say that  $K_{2,2}\subseteq G[\{x_6,x_7,x_8\}, Y\setminus \cup_{i=1}^{i=5} N_G(x_{i})]$, a contradiction again. For the case that $|Y'|\in \{30,31,32,33\}$, the proof is the same. Hence, the claim holds.
	\end{proof}	
If $\Delta\leq 6$, then it is clear that $K_{6,6}\subseteq \overline{G}$. Now, we may suppose that $\Delta=7$, and	w.l.g let $|N_G(x_1)=Y_1|=7$.   One can assume that $|N_G(x)\cap Y_1|= 1$ for at least three members of $X\setminus  \{x_1\}$. Otherwise, as $|Y_1|=|X\setminus \{x_1\}|=7$ and  $|N_G(x)\cap Y_1|\leq  1$ for  each member of $X\setminus  \{x_1\}$. Then it is easy to say that $K_{6,6}\subseteq \overline{G}[X,Y_1]$. Hence, w.l.g assume that $|N_G(x)\cap Y_1|= 1$ for each members of $\{x_2,x_3,x_4\}$. Hence as $\Delta=7$, one can check that $|\cup_{i=1}^{i=4} N_G(x_{i})|= 25$.  Therefore, we have $|\cup_{i=1}^{i=5} N_G(x_i)|\leq 32$, and by Claim \ref{c5}, we have  $K_{6,6}\subseteq \overline{G}$. So, we may suppose that $\Delta\in \{8,9,10\}$.  Now we consider the following cases.

\bigskip 
\noindent \textbf{Case 1.} $ \Delta=8$.	W.l.g suppose that $\Delta=|N_G(x_1)=Y_1|$. As $K_{2,2} \nsubseteq G$, one can suppose that there exists at least four members of $X\setminus\{x_1\}$ say $X'=\{x_2,\ldots, x_5\}$, such that $|N_G(x_i)\cap Y_1|= 1$ and $N_G(x_i)\cap Y_1\neq N_G(x_j)\cap Y_1$ for each $i,j\in \{2,\ldots,5\}$. Otherwise, one can check that $K_{7,5}\subseteq \overline{G}[X\setminus \{x_1\}, Y_1]$. Therefore, for each $y\in Y\setminus Y_1$, one can suppose that $|N_G(y)\cap X\setminus \{x_1\}|\geq 2$, Otherwise $K_{6,6}\subseteq \overline{G}$. So, as $|Y\setminus Y_1|=37$, we have $|E(G[X\setminus\{x_1\}, Y\setminus Y_1])|\geq 74$. Therefore, by  pigeon-hole principle  there is at least one member of $X\setminus \{x_1\}$ say $x$, such that $|N_G(x)|\geq 10$, a contradiction. Now, w.l.g suppose that $Y_1=\{y_1,\ldots, y_8\}$ and $x_iy_{i-1}\in E(G)$ for each $i=2,3,4,5$.  Set $Y'=\cup_{i=1}^{i=5} N_G(x_{i})$. Hence, it is easy to say that  $|Y'|\leq 36$. If $|Y'|\leq 35$, then proof is complete by Claim \ref{c5}.   So, let  $|Y'|= 36$. That is $\Delta=8=|N_G(x_i)|$  for each $i\in[5]$ and   $|N_G(x_i)\cap N_G(x_j)|=0 $ for each $i,j\in \{2,3,4,5\}$. Now consider $i=2,3$, as $K_{2,2}\nsubseteq G$, we have $|N_G(x_j)\cap  (N_G(x_i)\setminus \{y_{i-1}\})|\leq 1$ for each $j\in \{6,7\}$, hence one can say that $K_{6,6}\subseteq \overline{G}[X\setminus \{x_2,x_3\},  N_G(x_2)\cup N_G(x_2) \setminus \{y_1, y_2\}]$.

\bigskip 
\noindent \textbf{Case 2.} $ \Delta=9$. W.l.g suppose that $ N_G(x_1)=Y_1=\{y_1,y_2,\ldots,y_9\}$. Now, set $A$ as follow:
\[A=\{x\in X, ~~|N_G(x)|=\Delta=9\}\]
As $K_{2,2} \nsubseteq G$, we have $|N_G(x_i)\cap Y_1|\leq 1$. Hence,  one can say that $K_{6,3}\subseteq \overline{G}[X\setminus \{x_1,x\},Y_1]$ for each $x\in \{x_2,\ldots,x_8\}$. Therefore, by considering  the members of $A$, one can check that the following claim is true.
\begin{claim}\label{c6}
	If  $|N_G(x)\cap N_G(x')|=0$ for some  $x,x'\in A$, then $K_{6,6}\subseteq \overline{G}$.
\end{claim}	
Therefore, by  Claims \ref{c6}, we can prove the following claim. 
\begin{claim}\label{c7}
	If  $|N_G(x)\cap N_G(x')\cap N_G(x'')|=1$ for some  $x,x',x'' \in A$, then $K_{6,6}\subseteq \overline{G}$. 
\end{claim}	
\begin{proof}[Proof of Claim \ref{c7}]
	W.l.g assume that $x_1,x_2,x_3\in A$, $\{y_1\}=Y_1\cap Y_2\cap Y_3$, where $Y_i=N_G(x_i)$ for $i=1,2,3$. Since $K_{2,2} \nsubseteq G$,  for each $i\in[3]$ and each $x\in X\setminus\{x_1,x_2,x_3\}$ we have $|N_G(x)\cap Y_i|\leq 1$.  Therefore, as $|Y_i|=9$, and $|N_G(x)\cap Y_i|\leq 1$ for each $x\in X\setminus\{x_1,x_2,x_3\}$, it is easy to say that $K_{5,3}\subseteq \overline{G}[X\setminus \{x_1,x_2,x_3\}, Y_i\setminus \{y_1\}]$ for each $i\in[3]$. Therefore, we have $K_{5,6}\subseteq \overline{G}[X\setminus \{x_1,x_2,x_3\}, Y_1\cup Y_2\setminus \{y_1\}]$. So, as  $y_1\in Y_1\cap Y_2\cap Y_3$ and $K_{2,2} \nsubseteq G$, then  $N_G(x_3)\cap  (Y_1\cup Y_2\setminus \{y_1\})=\emptyset$. Therefore, $K_{6,6}\subseteq \overline{G}[X\setminus \{x_1,x_2\}, Y_1\cup Y_2\setminus \{y_1\}]$. Hence, the claim holds. 
\end{proof}	
Consider $|A|$. First suppose that $|A|\geq 5$, and w.l.g assume that $\{x_1,x_2,x_3,x_4,x_5\}\subseteq A$. Therefore, by Claim \ref{c6} and \ref{c7}, it can be said that     $|\cup_{i=1}^{i=5} N_G(x_{i})|= 35$. Hence, by Claim \ref{c5} the proof is complete.  So, we may assume that $|A|\leq 4$. Now, we verify the following two claims.
\begin{claim}\label{c8}
	If $|A|=4$, then $K_{6,6}\subseteq \overline{G}$. 
\end{claim}	
\begin{proof}[Proof of Claim \ref{c8}]
	W.l.g  assume that $A=\{x_1,x_2,x_3,x_4\}$. Therefore, by Claims \ref{c6} and \ref{c7}, one can checked that   $|\cup_{i=1}^{i=4} N_G(x_{i})|= 30$. Set $Y'=\cup_{i=1}^{i=4} N_G(x_{i})$. If there is a member of $X\setminus A$ say $x$, so that $3\leq |N_G(x)\cap Y'|$, then $|\cup_{i=1}^{i=5} N_G(x_{i})\cup N_G(x)|\leq 35$ and the proof is complete by Claim \ref{c5}.
	
	Hence, we may suppose that  $|N_G(x)\cap Y'|\leq 2$ for each $x\in X\setminus A$. So as $|Y'|=30$, one can checked that  $K_{4, 22}\subseteq \overline{G}[X\setminus A, Y']$. W.l.g let  $K_{4, 22}\cong  \overline{G}[X\setminus A, Y'']$, where $Y''\subseteq Y'$ and $|Y''|=22$. Therefore, it is easy to checked that there is at least two members of $A$ say $x_{i_1} x_{i_2}$, such that $|(N_G(x_{i_1})\cup N_G(x_{i_1}))\cap Y''|\leq 16$. W.l.g let $i_1=1, i_2=2$. So, we have  $K_{6, 6}\cong  \overline{G}[X\setminus \{x_3,x_4\}, Y'']$. Hence, the claim holds. 
\end{proof}
\begin{claim}\label{c9}
	If $|A|=3$, then $K_{6,6}\subseteq \overline{G}$. 
\end{claim}	
\begin{proof}[Proof of Claim \ref{c9}]
	W.l.g  suppose that $ A=\{x_1,x_2,x_3\}$. Therefore, by Claims \ref{c6} and \ref{c7}, it can be said that    $|\cup_{i=1}^{i=3} N_G(x_{i})|= 24$. Set $Y'=\cup_{i=1}^{i=3} N_G(x_{i})$. Suppose that there exists a vertex of $X\setminus A$ say $x$, such that $|N_G(x)\cap Y'|\geq 2$, then we have $|\cup_{i=1}^{i=4} N_G(x_{i})\cup N_G(x)|\leq 30$. W.l.g assume that $x=x_4$.  If  $|\cup_{i=1}^{i=4} N_G(x_{i})|\leq 27$, then the proof is complete by Claim \ref{c5}.  So, suppose that  $28\leq |\cup_{i=1}^{i=4} N_G(x_{i})|\leq 30$. Let $|\cup_{i=1}^{i=4} N_G(x_{i})|= 30$. Set $X'=\{x_5,x_6,x_7,x_8\}$. In this case, one can suppose that  $|N_G(x)\cap (Y\setminus \cup_{i=1}^{i=4} N_G(x_{i}))|\geq 6$ for each $x\in X'$. Otherwise, the proof is complete by Claim \ref{c5}. Therefore, as $|X'|=4$ and $|Y\setminus \cup_{i=1}^{i=4} N_G(x_{i}))|= 15$, then one can checked that $K_{2,2}\subseteq G$, a contradiction. For the case that  $|\cup_{i=1}^{i=4} N_G(x_{i})|=28, 29$, the proof is the same. 
	
	 So, let  $|N_G(x)\cap Y'|\leq 1$ for each $x\in X\setminus A$. Therefore, it is clear that  $K_{5, 19}\subseteq \overline{G}[X\setminus A, Y']$. W.l.g let  $K_{5, 19}\cong  \overline{G}[X\setminus A, Y'']$, where $Y''\subseteq Y'$ and $|Y''|=19$. Therefore, one can said  that there is at least one member of $A$ say $y$, so that $|N_G(y)\cap Y''|\leq 9$. W.l.g let $y=y_1$. So  $K_{6, 6}\cong  \overline{G}[X\setminus \{x_2, x_3\}, Y'']$. Hence, the claim holds. 
\end{proof}

Hence, by Claims \ref{c8} and \ref{c9}, one can suppose that $|A|\leq 2$.  First, assume  that $|A|=2$ and w.l.g suppose that $A=\{x_1,x_2\}$. By Claim \ref{c7}, we have $|\cup_{i=1}^{i=2} N_G(x_{i})|= 17$. For $i=1,2$, set $Y_i=N_G(x_i)$. By Claim \ref{c7} w.l.g let $y_1\in Y_1\cap Y_2$. Set $X'=X\setminus A$. Suppose that  there is at least two vertices of $X'$ say $x_3,x_4$,  such that $|N_G(x)\cap Y_i\setminus \{y_1\}|=0$ for at least one $i\in[2]$. W.l.g let  $|N_G(x)\cap Y_1\setminus \{y_1\}|=0$, therefore as $|N_G(x)\cap Y_1\setminus \{y_1\}|\leq 1$ and $|Y_1|=9$, one can say that  $K_{6, 4}\subseteq \overline{G}[X\setminus A, Y_1\setminus \{y_1\}]$. Also, one can check that  $K_{6, 2}\subseteq \overline{G}[X\setminus A, Y_2\setminus \{y_1\}]$, hence $K_{6,6}\subseteq \overline{G}[X\setminus A, Y_1\cup Y_2\setminus \{y_1\}]$. Therefore for any $i\in \{1,2\}$, we may suppose that $|N_G(x)\cap Y_i\setminus \{y_1\}|=1$  for at least five members of $X'$. Hence as $|X'|=6$, it is clear that there is at least three member of $X'$ say $\{x_3,x_4,x_5\}$, so that for any $x\in \{x_3,x_4,x_5\}$, we have $|N_G(x)\cap( Y_1 \cup Y_2\setminus \{y_1\})|= 2$.  Therefore, as $|N_G(x)|\leq 8$ for each $i=3,4,5$, one can check that   $|\cup_{i=1}^{i=5} N_G(x_{i})|\leq 17+18=35$.  Hence, the proof is complete by Claim \ref{c5}.

Now, let  $|A|=1$ and w.l.g let $A=\{x_1\}$. In this case, one can say that there exist at least five vertices of $X\setminus \{x_1\}$ say $X''=\{x_2,x_3,x_4,x_5, x_6\}$, such that $|N_G(x)\cap Y_1|= 1$ for each $x\in X''$. Otherwise, as $|Y_1|=9$ and $|N_G(x)\cap Y_1|\leq 1$ for each $x\in X\setminus\{x_1\}$, then one can say that  $K_{6, 6}\cong  \overline{G}[X\setminus \{x_1\}, Y_1]$. Therefore, there is at least one vertex of $X''$ say $x_2$, so that $|N_G(x)|=8$. Otherwise, we have $|\cup_{i=2}^{i=6} N_G(x_{i})|\leq 35$ and the proof is complete by Claim\ref{c5}. W.l.g assume  that $Y_2=N_G(x_2)\cap Y\setminus Y_1$ and $|Y_2|=7$. Therefore, one can say that there is at least two vertices of $X''\setminus \{x_2\}$ say $\{x_3,x_4\}$, so that  for each $x\in\{x_3,x_4\}$, we have $|N_G(x)\cap Y_2|= 1$, otherwise as $|Y_2|=7$ and $|N_G(x)\cap Y_1|\leq 1$ for each $x\in X\setminus\{x_1\}$, then one can say that  $K_{6, 6}\cong  \overline{G}[X\setminus \{x_2\}, Y_2]$. Now, one can check that $|\cup_{i=1}^{i=5} N_G(x_{i})|\leq 9+7+6+6+7=35$, and the proof is complete by Claim\ref{c5}. 
	
\bigskip 
\noindent \textbf{Case 3.} $ \Delta=10$. W.l.g let $N_G(x_1)=Y_1=\{y_1,\ldots,y_{10}\}$. Therefore by $K_{2,2} \nsubseteq G$ it is clear to say that $K_{6,4}\subseteq \overline{G}[X\setminus \{x_1,x\},Y_1]$ for each $x\in X\setminus\{x_1\}$.  Let there is a member of $ X\setminus\{x_1\}$ say $x_2$, so that $| N_G(x_2)\cap (Y\setminus Y_1)=Y_2|=8$.  Therefore, as  $K_{2,2} \nsubseteq G$, we have $|N_G(x_i)\cap Y_2|\leq 1$. Hence, since $|Y_2|=8$ and $|X\setminus \{x_1,x_2\}|=2$, one can say that $K_{6,2}\subseteq \overline{G}[X\setminus \{x_1,x_2\},  Y_2]$. So, $K_{6,6}\subseteq \overline{G}[X\setminus \{x_1,x_2\},  Y_1\cup Y_2]$. Now, one can suppose that $| N_G(x)\cap (Y\setminus Y_1)|\leq 7$ for any member of $X\setminus \{x_1\}$. Hence:
\begin{claim}\label{c10}
	Suppose that	$| Y'=N_G(x)\cap (Y\setminus Y_1)|= 7$. If either $| N_G(x')\cap Y'|= 0$ for one $x'\in X\setminus \{x_1,x\}$, or $| N_G(x')\cap N_G(x'')\cap Y'|= 1$ for some $x',x''\in X\setminus \{x_1,x\}$,  then $K_{6,6}\subseteq \overline{G}$.
\end{claim}	
\begin{proof}[Proof of Claim \ref{c10}]
	W.l.g let 	$| Y'=N_G(x_2)\cap (Y\setminus Y_1)|= 7$.  Also, w.l.g let $| N_G(x_3)\cap Y'|= 0$. Therefore,  Since $K_{2,2} \nsubseteq G$, so $|N_G(x_i)\cap (Y_1\cup Y')|\leq 2$ for each $i\in \{3,4,5,6,7,8\}$. As $|Y_1\cup Y'|=17$, $| N_G(x_3)\cap Y'|= 0$, and $|N_G(x_i)\cap Y_1\cup Y'|\leq 2$, one can say that $| \cup_{i=3}^{i=8}( N_G(x_i)\cap Y_1\cup Y')|\leq 11$, which means that $K_{6,6}\subseteq \overline{G}[X\setminus \{x_1,x_2\},Y_1\cup Y_2]$. For the case that $| N_G(x')\cap N_G(x'')\cap Y'|= 1$ for some $x',x''\in X\setminus \{x_1,x\}$, the proof is the same. Hence, the claim holds.
\end{proof}	
Set $M$ as follow:
\[M= \{x\in X\setminus \{x_1\}, ~~| N_G(x)\cap (Y\setminus Y_1)|= 7\}.\]
By considering $M$, we have:
\begin{claim}\label{c11}
	If $|M|\neq 0$, then we have $K_{6,6}\subseteq \overline{G}$.
\end{claim}	
\begin{proof}[Proof of Claim \ref{c11}]
	W.l.g let $x_2\in M$, and $N_G(x_2)\cap (Y\setminus Y_1)=Y_2=\{y_{11}, y_{12} \ldots, y_{17}\}$.	If $|M|\geq 5$, then by Claim \ref{c10}, it can be said that $|\cup_{x_j\in M'} N_G(x_j)|\leq 35$, where $M'\subseteq M$ and $|M'|=5$. Hence, the proof is complete by Claim \ref{c5}. Now, assume that $|M|=i$,  and w.l.g suppose that $M=\{x_2,x_3,\ldots, x_{i+1}\}$, where $i\in \{1,2,3,4\}$. Also if 	 $|M|\leq 2$, then it can be said that $|\cup_{x_j\in M''} N_G(x_j)|\leq 35$, where $M''\subseteq X\setminus M$ and $|M''|=5$. Hence, the proof is complete by Claim \ref{c5}. Now assume that $|M|=i$, where $i\in \{3,4\}$.
	
	By Claims \ref{c10}, for the case that $i=4$, we have $|\cup_{j=1}^{j=5} N_G(x_j)|= 10+7+6+5+4=32$. Hence, the proof is complete by Claim \ref{c5}. Also for the case that $i=3$, by Claims \ref{c10},  we have $|\cup_{j=1}^{j=4} N_G(x_j)|= 10+7+6+5=28$. Therefore, we have  $|\cup_{j=1}^{j=5} N_G(x_j)|\leq 34$.  Hence, the proof is complete by Claim \ref{c5}.
	 
\end{proof}	
Now, by Claim \ref{c11}, let $|M|=0$, that is $| N_G(x)\cap (Y\setminus Y_1)|\leq  6$ for each $x\in X\setminus\{x_1\}$. In this case, the proof is complete by Claim \ref{c5}.

Hence, by Cases 1, 2, 3, the upper bound holds. 	Thus,  $BR_8(K_{2,2}, K_{6,6}) =45$ and theorem holds.
\end{proof}

	\begin{proof}[\bf Proof of Theorem \ref{M.th}]
	By combining Theorems  \ref{th1}, \ref{th2},  and \ref{th3},  Theorem \ref{M.th} is holds.
	\end{proof}
 
	\bibliographystyle{spmpsci} 
	\bibliography{BI}
\end{document}